\documentclass[12pt]{amsart}
\usepackage{amsthm}
\usepackage{amsmath}
\usepackage{amssymb}
\usepackage{comment}
\usepackage{enumitem}
\usepackage{amsfonts}
\usepackage{mathtools}
\usepackage{setspace}
\usepackage[dvipsnames]{xcolor}
\usepackage[hidelinks]{hyperref}
\usepackage{dynkin-diagrams}
\hypersetup{
    colorlinks=true,
    linkcolor=blue,
    urlcolor=red,
    citecolor=Green
    }
\usepackage{subfigure}

\usepackage[utf8]{inputenc}

\advance\paperheight1.00cm
\advance\textheight1.00cm
\advance\vsize1.00cm
\advance\topskip-0.5cm
\advance\voffset-0.5cm

\usepackage{tikz,tikz-cd}
\usepackage{todonotes}
\DeclareTextFontCommand{\emph}{\color{blue}\em}

\newcommand{\precdot}{\mathrel{\ooalign{$\prec$\cr
  \hidewidth\raise0.001ex\hbox{$\cdot\mkern0.6mu$}\cr}}}

\newtheorem{theorem}{Theorem}[section]
\newtheorem{proposition}[theorem]{Proposition}
\newtheorem{conjecture}[theorem]{Conjecture}

\newtheorem{lemma}[theorem]{Lemma}
\newtheorem{corollary}[theorem]{Corollary}
\newtheorem{problem}[theorem]{Problem}

\theoremstyle{definition}
\newtheorem{ex}[theorem]{Example}
\newtheorem{definition}[theorem]{Definition}

\theoremstyle{remark}
\newtheorem*{remark}{Remark}

\newcommand{\R}{\mathbb{R}}
\DeclareMathOperator{\VF}{VF}
\DeclareMathOperator{\At}{At}
\DeclareMathOperator{\Co}{Co}

\title{Lattice Bruhat intervals are polytopal}
\author{Yibo Gao}
\address{Beijing International Center for Mathematical Research, Peking University, Beijing 100871, China}
\email{gaoyibo@bicmr.pku.edu.cn}
\author{Tao Gui}
\address{Institute for Theoretical Sciences, Westlake University, No.\ 600 Dunyu Road, Xihu District, Hangzhou, Zhejiang Province 310030, China}
\email{guitao18@mails.ucas.ac.cn}

\date{}

\begin{document}
\pagestyle{plain}
\begin{abstract}
For a Bruhat interval $[x,y]$ in an arbitrary Coxeter group which is a lattice, we give a type-uniform construction of a convex polytope whose complete face lattice realizes $[x,y]$. This gives a polytopal realization of the coincidence between classical Kazhdan--Lusztig polynomials for lattice Bruhat intervals and Stanley's toric $g$-polynomials. Consequently, the Kazhdan--Lusztig polynomial of a lattice Bruhat interval cannot have internal zeros, satisfies a multiplicative strengthening of monotonicity, and obeys the generalized upper bound theorem for polytopes. For lattice Bruhat intervals in finite Weyl groups, we explain the relation with the toric geometry of Richardson varieties and derive formulas for their local and global intersection cohomology, as well as a Boolean criterion for rational smoothness. We conjecture that the coefficients of these Kazhdan--Lusztig polynomials form an $M$-sequence.
\end{abstract}
\maketitle

\section{Introduction}
In 1979 Kazhdan and Lusztig introduced the remarkable polynomials $P_{x,y}(q)$ for comparable elements of a Coxeter group \cite{KL1979}; for Weyl groups, their coefficients encode the dimensions of local intersection cohomology groups of Schubert varieties \cite{KL1980}. The purely combinatorial definition of these polynomials inspired Stanley to introduce toric $g$-polynomials for convex polytopes \cite{Stanley1987}, which have a similar interpretation in the intersection cohomology of toric varieties when the polytopes are rational \cite{DenefLoeser1991,Fieseler1991}. The marvelous positivity theorems of Elias--Williamson \cite{EliasWilliamson2014} and Karu \cite{Karu2004} extend the positivity of the coefficients of the Kazhdan--Lusztig polynomials and toric $g$-polynomials in the geometric settings to arbitrary Coxeter groups and arbitrary convex polytopes, respectively. 

Nevertheless, the two families of polynomials have quite different numerical behavior. A famous theorem of Polo \cite{Polo1999} states that any polynomial in $1+q\mathbb Z_{\geq0}[q]$ appears as a Kazhdan--Lusztig polynomial in some symmetric group. In contrast, toric $g$-polynomials of polytopes satisfy strong restrictions: they cannot have internal zeros and are conjectured to have coefficients forming an $M$-sequence (see \cite[Theorem 1.4 and Section 1.2]{Braden2006}).

The Kazhdan--Lusztig--Stanley theory provides a unified framework for these constructions \cite{Stanley1992}. Its input is a ranked poset $P$ together with an auxiliary element (known as a $P$-kernel) in its incidence algebra.  Classical Kazhdan--Lusztig polynomials arise from the Kazhdan--Lusztig $R$-polynomial kernel of the Bruhat intervals in Coxeter groups, while toric $g$-polynomials arise from the Eulerian kernel $(q-1)^{\operatorname{rk}(a,b)}$ on face posets of polytopes (see \cite{Stanley1992,Proudfoot2018}). Note that the face poset of any convex polytope is a lattice. It is known that if a Bruhat interval is a lattice, these two kernels agree on every subinterval. Hence the Kazhdan--Lusztig polynomials in this case coincide with toric $g$-polynomials \cite{Brenti-combinatorial}. Our main result shows that this agreement comes from an actual convex polytope, building a bridge for the two settings of the Kazhdan--Lusztig--Stanley theory. 

\begin{theorem}\label{thm:main}
Let $(W,S)$ be any Coxeter system and let $x<y$ in the Bruhat order. If the Bruhat interval $[x,y]$ is a lattice, then there is a convex polytope $L_{x,y}$ of dimension $\ell(x,y)-1$ whose complete face lattice is isomorphic to $[x,y]$.
\end{theorem} 

Here the complete face lattice includes both the empty face and the whole polytope. They correspond to $x$ and $y$, respectively, under the isomorphism. This uniform construction will be provided in Section~\ref{sec:polytope}.

We explain some related geometry. Suppose that $W$ is a finite Weyl group of a complex semisimple algebraic group $G$ and that the Bruhat interval $[x,y]\subset W$ is a lattice. In this case, Theorem~\ref{thm:main} was recently proved by Bossinger--Gorsky--Simental \cite[Theorem~1.8 and Corollary~6.8]{BossingerGorskySimental2026} for finite Weyl groups based on earlier work of Gorsky--Kim--Sherman-Bennett in type $A$ \cite{GorskyKimShermanBennett} using toric Richardson varieties. 

More precisely, let $\overline{R}_{x,y}$ be the closed Richardson variety in the complete flag variety, and let $R_{x,y}$ be its open Richardson stratum. For finite Weyl groups, the lattice condition is equivalent to $R_{x,y}$ being an algebraic torus and to $\overline{R}_{x,y}$ being toric (see Proposition~\ref{prop:toric-Richardson}). The acting torus may be unexpected, since it need not be a subtorus of the maximal torus $T$ of the algebraic group $G$ and its rank can be significantly larger than that of the maximal torus of $G$ (see \cite[Section 5.3]{GorskyKimShermanBennett}). The faces of the moment polytope $Q_{x,y}$ of $\overline{R}_{x,y}$ are indexed by subintervals of $[x,y]$, and a vertex figure of $Q_{x,y}$ at the vertex indexed by $[x,x]$ then has complete face lattice isomorphic to $[x,y]$ \cite[Corollary~6.8]{BossingerGorskySimental2026}. In other words,  in the geometric setting, a rational polytopal realization of $[x,y]$ is obtained as the moment polytope of a projective toric variety $X_{x,y}$, which is a torus-invariant divisor in a toric modification of the closed Richardson variety $\overline{R}_{x,y}$. When $\overline{R}_{x, y}$ is toric and the maximal torus $T$ has a dense orbit on $\overline{R}_{x, y}$, the moment polytope $Q_{x,y}$ can be realized as a Bruhat interval polytope defined by Kodama--Williams \cite{KodamaWilliams2014} and studied by Tsukerman--Williams \cite{TsukermanWilliams2015}. In type A, Lee--Masuda--Park \cite[Theorem~5.1]{LeeMasudaPark2021} showed that a Bruhat interval polytope is toric (that is, its dimension is equal to the length of the Bruhat interval) precisely when every subinterval is realized as a face. In particular, in this case its face lattice is the interval poset of the underlying Bruhat interval.

This picture allows us to express the local and global intersection cohomology of toric $\overline{R}_{x,y}$ in terms of the toric Kazhdan--Lusztig--Stanley polynomials of the interval poset $\operatorname{Int}([x,y])$ (see Theorem~\ref{thm:Richardson-IH}). We also give several equivalent conditions for the rational smoothness of the closed Richardson variety $\overline{R}_{x,y}$ (see Theorem~\ref{thm:rational-smoothness}). In particular, we show that for a lattice Bruhat interval in a finite Weyl group,
\[
\overline{R}_{x,y}\text{ is rationally smooth} \quad\Longleftrightarrow\quad P_{x,y}(q)=1 \quad\Longleftrightarrow\quad [x,y]\text{ is Boolean}.
\]
In this case, the moment polytope $Q_{x,y}$ of the closed Richardson variety $\overline{R}_{x,y}$ is combinatorially a cube and the ordinary cohomology Poincar\'e polynomial of $\overline{R}_{x,y}$ is $(1+q)^{\ell(x,y)}$. See Section~\ref{sec:geometry} for precise statements and proofs. 

In the geometric setting, we can interpret the classical Kazhdan--Lusztig polynomial as a local intersection cohomology Poincar\'e polynomial of a toric variety (see Corollary~\ref{cor:local-toric-KL} and the following paragraph).

We now return to arbitrary Coxeter groups. Write $g(P,q)$ for Stanley's toric $g$-polynomial of a polytope $P$, with the convention extending the $g$-vector of a simplicial polytope. Then for the polytope $L_{x,y}$ in Theorem~\ref{thm:main}, we have 
\begin{equation*}
 P_{x,y}(q)=g(L_{x,y}^{*},q),
\end{equation*}
where $L_{x,y}^{*}$ denotes the polar dual after translating the origin to the interior of $L_{x,y}$. We have the following theorem.
\begin{theorem}\label{thm:KL-applications}
Let $[x,y]$ be a lattice Bruhat interval in any Coxeter group. Let $n=\ell(x,y)$ and let $m$ be the number of coatoms in $[x,y]$. Then the following statements hold. 

\begin{enumerate}
\item The coefficients of $P_{x,y}(q)$ have no internal zeros. In other words, 
$$
[q^i]P_{x,y}(q)=0  \implies [q^{i+1}]P_{x,y}(q)=0,
$$
where $[q^i]P_{x,y}(q)$ denotes the coefficient of $q^i$ in $P_{x,y}(q)$. 
\item For every $x\leq z\leq y$,
\begin{equation*}
 P_{x,y}(q)\geq P_{x,z}(q)P_{z,y}(q)
\end{equation*}
coefficient-wise.
\item For $1\leq i\leq\lfloor(n-1)/2\rfloor$, 
\begin{equation*}
 [q^i]P_{x,y}(q)\leq\binom{m-n+i-1}{i}.
\end{equation*}
\item The equality $P_{x,y}(q)=1$ holds if and only if $[x,y]$ is Boolean.
\end{enumerate}
\end{theorem}

Theorem~\ref{thm:KL-applications} uses the lattice hypothesis combined with Theorem~\ref{thm:main} to obtain  restrictions coming from results on convex polytopes. The corresponding results are due to Kalai, Braden--MacPherson, Bayer, and the theory of combinatorial intersection cohomology of fans (see \cite[Theorems~1.2--1.4]{Braden2006} and the references therein). We recall the kernel formalism of the Kazhdan--Lusztig--Stanley theory in Section~\ref{sub:KLS}, and prove Theorem~\ref{thm:KL-applications} in Section~\ref{sec:KL}. In particular, lattice Bruhat intervals provide a class for which the combinatorial invariance of Kazhdan--Lusztig polynomials follows directly from the defining recursion.  For the broader combinatorial invariance conjecture, see \cite{BlundellEtAl2022,BarkleyGaetzLam}.

Theorem~\ref{thm:KL-applications} means that the Kazhdan--Lusztig polynomials for lattice Bruhat intervals are much more restricted, and they satisfy a  stronger inequality than the usual monotonicity $P_{u,v}(q)\ge P_{w,v}(q)$ for $u\leq w\leq v$ which holds for Kazhdan--Lusztig polynomials of general Bruhat intervals. These restrictions in Theorem~\ref{thm:KL-applications} cannot hold for arbitrary Bruhat intervals, in view of Polo's realization theorem \cite{Polo1999}. However, a full characterization of what polynomials occur as toric $g$-polynomials of polytopes is currently not known. It is a well-known conjecture of Stanley that coefficients of a toric $g$-polynomial of any polytope form an $M$-sequence \cite{Stanley1987,Braden2006}. Applying this conjecture to the specific family of polytopes in Theorem~\ref{thm:main}, we have the following conjecture, which is stronger than the coefficient restrictions (1) and (3) in Theorem~\ref{thm:KL-applications}. 

\begin{conjecture} \label{conj-TKL is M}
If $[x,y]$ is a lattice Bruhat interval in any Coxeter group, then the coefficients of $P_{x,y}(q)$ form an $M$-sequence.
\end{conjecture}

Recall that a finite sequence $(c_0,c_1,\ldots,c_d)$ is an \emph{$M$-sequence} if there is a commutative standard graded Artinian algebra $A=\bigoplus_{i\geq0}A_i$ over a field, with $A_0$ equal to the field and $A$ generated by $A_1$, such that $c_i=\dim A_i$. The algebraic issue is that intersection cohomology generally has no natural ring structure. Even in the geometric setting, realizing the coefficients of $P_{x,y}(q)$ as dimensions of local intersection cohomology groups of projective toric varieties does not supply an algebra generated in degree one. This leads to a more structural version of the question.

\begin{problem}\label{prob:KL-algebra}
For each lattice Bruhat interval $[x,y]$ in a Coxeter group, give a natural construction of a commutative standard graded Artinian algebra $A_{x,y}$ with
\[
 \operatorname{Hilb}(A_{x,y};q)=P_{x,y}(q).
\]
Can such a construction be made compatible with passage to subintervals?
\end{problem}

\section{Preliminaries}\label{sub:prelim}
\subsection{Root systems for Coxeter groups}\label{sub:root-systems}
Let $(W,S)$ be a Coxeter system, with length function $\ell$, Coxeter matrix $(m_{st})_{s,t\in S}$, and reflections $T=\{wsw^{-1}:w\in W,\ s\in S\}.$ Write $\ell(x,y)=\ell(y)-\ell(x)$. We allow $W$ to be infinite, but assume that $S$ is finite. We use the real realization described in \cite[Section~1]{Patimo}. Let $\mathfrak h$ be a finite-dimensional real vector space of minimal dimension admitting linearly independent sets 
\[
\Delta=\{\alpha_s:s\in S\}\subseteq\mathfrak h^*,\qquad \Delta^\vee=\{\alpha_s^\vee:s\in S\}\subseteq\mathfrak h
\]
which are called \emph{simple roots} and \emph{simple coroots}, respectively, such that
\[
\alpha_s(\alpha_t^\vee)=-2\cos\left(\frac{\pi}{m_{st}}\right).
\]
Here $m_{ss}=1$ and $\cos(\pi/\infty)=1$. The formulas
\[
s(v)=v-\alpha_s(v)\alpha_s^\vee,\qquad s(\lambda)=\lambda-\lambda(\alpha_s^\vee)\alpha_s
\]
define dual actions of $W$ on $\mathfrak h$ and $\mathfrak h^*$. This realization is reflection-faithful: the action is faithful, and the fixed-point subspace $\mathfrak h^w$ has codimension one precisely when $w\in T$. The associated sets of \emph{roots} and \emph{coroots} are
\[
\Phi=\{w\alpha_s:w\in W,\ s\in S\},\qquad\Phi^\vee=\{w\alpha_s^\vee:w\in W,\ s\in S\}.
\]
Define the \emph{positive roots} and \emph{negative roots} to be
\[
\Phi^+=\Phi\cap\sum_{s\in S}\mathbb R_{\geq0}\alpha_s,\qquad\Phi^-=-\Phi^+.
\]
Similarly write $(\Phi^\vee)^+$ and $(\Phi^\vee)^-$ for the positive and negative coroots. Each reflection $t\in T$ determines a unique positive root $\alpha_t$ and positive coroot $\alpha_t^\vee$. Explicitly, if $t=wsw^{-1}$ and $w\alpha_s\in\Phi^+$, then $\alpha_t=w\alpha_s$ and $\alpha_t^\vee=w\alpha_s^\vee$.

\subsection{Lattice Bruhat intervals}\label{sub:lattice-Bruhat}
For $x\leq y$ in the Bruhat order, the \emph{directed Bruhat graph} $\Gamma(x,y)$ contains $[x,y]$ as vertices, and $w\rightarrow wt$ as edges, where $x\leq w<wt\leq y$, $t\in T$. The \emph{left coroot label} of this edge $w\rightarrow wt$ is $w\alpha_t^{\vee}$ which is positive. Write $\hat\Gamma(x,y)$ for its underlying undirected graph and $E_{x,y}$ for its edge set. A subset of edges $F\subseteq E_{x,y}$ is called \emph{diamond-closed} if whenever it contains two adjacent edges in a $4$-cycle of $\hat\Gamma(x,y)$, it also contains the other two. The \emph{diamond closure} of $F$ is the smallest diamond-closed subset containing $F$. And $F\subseteq E_{x,y}$ is \emph{diamond-generating} if its diamond closure equals $E_{x,y}$. Write $g_{x,y}$ for the minimum cardinality of a diamond-generating set. See \cite{Patimo} and \cite{BarkleyGaetzLam} for more details on diamond-generating sets.

Denote the atoms of $[x,y]$ by $\At(x,y):=\{z:x\lessdot z\leq y\}$ and the coatoms by $\Co(x,y):=\{z:x\leq z\lessdot y\}$. It is easy to see that any maximal chain is a diamond-generating set, and that $\{(x,a):a\in\At(x,y)\}$ and $\{(c,y):c\in\Co(x,y)\}$ also diamond generate.

An edge $w\rightarrow wt$ is called \emph{long} if $\ell(wt)-\ell(w)>1$. An interval $[x,y]$ is \emph{$S_3$-free} if it contains no subinterval isomorphic to the Bruhat order on $S_3$, also called a \emph{$2$-crown}. Here is a summary of some of the known characterizations of a lattice Bruhat interval.

\begin{theorem}\label{thm:lattice-characterization}
The following are equivalent for a Bruhat interval $[x,y]$.
\begin{enumerate}
\item The Bruhat interval $[x,y]$ is a lattice.
\item The Bruhat interval $[x,y]$ is $S_3$-free.
\item The Bruhat graph $\Gamma(x,y)$ does not contain a long edge.
\item The $R$-polynomial satisfies $R_{x,y}(q)=(q-1)^{\ell(x,y)}$.
\item Every maximal chain of $[x,y]$ is a minimum-size diamond generating set.
\end{enumerate}
\end{theorem}
The implication (1)$\Rightarrow$(2) is immediate with the converse being an unpublished result of Dyer (see also \cite[Section~1]{GorskyKimShermanBennett}). The equivalence of (2) and (3) follows from \cite[Proposition~3.3]{Dyer-Bruhat-graph}, while the equivalence of (2) and (4) is Brenti's criterion \cite[Theorem~6.3]{Brenti-combinatorial}. Condition (5) is saying that $g_{x,y}=\ell(x,y)$ while \cite{BarkleyGaetzLam} gives $g_{x,y}=-[q^{\ell(x,y)-1}]R_{x,y}(q)=(-1)^{\ell(x,y)-1}[q]R_{x,y}(q)$. The equivalence of (4) and (5) is then seen from \cite[Theorem~6.3]{Brenti-combinatorial}.

We say that $[x,y]$ is a \emph{lattice Bruhat interval} if the Bruhat interval $[x,y]$ is a lattice.

\subsection{Face structures of polytopes and cones}\label{sub:normal-fan}
For a polytope $P\subset V$ where $V$ is a real vector space, its \emph{faces} are given by $F_f=\{v\in P:f(v)=\max_{p\in P}f(p)\}$, where $f\in V^*$ is called a \emph{supporting vector} and $F_f$ is the corresponding \emph{supporting face}. The \emph{complete face lattice} $\mathcal{F}(P)$ of a polytope $P$ is the poset of all faces of $P$ ordered by inclusion, together with the $(-1)$-dimensional face $\emptyset$ which is the minimum of this poset and $P$ itself which is the maximum. It is easy to see that this poset is indeed a lattice, with the meet operation given by intersection. 

In this paper, a \emph{cone} always means a convex polyhedral cone. We say that a cone $C\subset V$ is \emph{pointed} if there exists $f\in V^*$ such that $f(v)>0$ for all $v\in C\setminus\{0\}$. For a pointed cone $C$, its \emph{faces} are given by $F_f=\{v\in C:f(v)=0\}$ where $f\in V^*$ satisfies that $f(v)\geq0$ for all $v\in C$. Analogously, $f$ is called a \emph{supporting vector} and $F_f$ is the corresponding \emph{supporting face}. The \emph{complete face lattice} $\mathcal{F}(C)$ of a pointed cone $C$ is the poset of all faces of $C$ ordered by inclusion, together with maximum $C$ and minimum $\{0\}$.

Given a polytope $P\subset V$, we can construct a pointed cone $C(P)=\{(tv,t):v\in P,t\in\R_{\geq0}\}\subset V\oplus \R$. They have the same face structure. Conversely, given a pointed cone $C$, we can take its \emph{vertex figure} $\VF(C)=\{v\in C:f(v)=1\}$ for some $f\in V^*$ such that $f(v)>0$ for all $v\in C\setminus\{0\}$. This is a polytope with the same face structure as $C$. Only the combinatorial type of $\VF(C)$ matters for our purposes, as it is independent of the chosen strictly positive functional.

Similarly, for a polytope $P$ and a vertex $p$, we can also take the \emph{vertex figure} of $P$ at $p$, defined as $\VF_p(P) =\{v\in P:f(v)=f(p)-\epsilon\}$ where $f$ is a supporting vector of $p$, and $\epsilon>0$ is small enough. The complete face lattice of $\VF_p(P)$ is then the principal order filter in $\mathcal{F}(P)$ generated by $p$.

After translating a polytope $P$ and working in its affine span, we may assume, for combinatorial purposes, that $P\subset V$ is full-dimensional and contains the origin in its interior. Its \emph{polar dual} is \[
P^*:=\{f\in V^*:f(v)\leq1\text{ for all }v\in P\}.
\]
The complete face lattice of $P^*$ is dual to that of $P$. 

\subsection{Kazhdan--Lusztig--Stanley theory} \label{sub:KLS}

We briefly recall the definitions in the Kazhdan--Lusztig--Stanley theory, keeping track of left and right conventions. We follow the convention and terminology of \cite{Proudfoot2018}, which we recommend as a wonderful survey in this theory.

Let $\mathcal L$ be a finite ranked poset. For $a\le b$, let $r(a,b)$ denote their rank difference. The \emph{incidence algebra} $I(\mathcal L)$ of $\mathcal L$ consists of families
$$
f=(f_{a,b}(q))_{a\leq b}, \quad \text{ with } f_{a,b}(q)\in \mathbb{Z}[q],
$$
whose ring structure is given by convolution
\[
 (fg)_{a,c}(q)=\sum_{a\leq b\leq c}f_{a,b}(q)g_{b,c}(q).
\]
The identity element is $\delta\in I(\mathcal L)$, with $\delta_{a,a}=1$ and $\delta_{a,b}=0$ for $a<b$. 

For $f\in I(\mathcal L)$ satisfying $\deg f_{a,b}\leq r(a,b)$, we  define an involution $f \mapsto \overline f$ by
\[
 \overline f_{a,b}(q)=q^{r(a,b)}f_{a,b}(q^{-1}).
\]

An element $\kappa\in  I(\mathcal L)$ is called a \emph{kernel} if  $\kappa_{a,a}=1$ for all $a \in \mathcal L$ and $\overline\kappa=\kappa^{-1}$. If $\kappa\in  I(\mathcal L)$ is a kernel, then there are unique elements $f$ and $g$ with degree less than ${r(a,b)}/2$ for $a<b$ and  $f_{a,a}=g_{a,a}=1$ for all $a \in \mathcal L$, satisfying
\begin{equation*}
 \overline f=\kappa f,\qquad \overline g=g\kappa.
\end{equation*}
We refer to $f$ as the \emph{right KLS function} associated with $\kappa$ and to $g$ as the \emph{left KLS function} associated with $\kappa$. Note that reversing the order on the poset exchanges $f$ and $g$.

For a Bruhat interval $[x,y]$, the Kazhdan--Lusztig $R$-polynomials form a kernel and its right KLS function $P=(P_{a,b}(q))$ gives the classical Kazhdan--Lusztig polynomials \cite{KL1979}. If $\mathcal L$ is the face lattice of a polytope (or more generally an Eulerian poset), then
\begin{equation*}
 \epsilon_{a,b}(q)=(q-1)^{r(a,b)} \text{ for all } a\le b
\end{equation*}
is a kernel \cite[Proposition 7.1]{Stanley1992}, known as the \emph{Eulerian kernel}. 

\section{The polytopal realization}\label{sec:polytope}
Let $[x,y]$ be a Bruhat interval of $(W,S)$. Our main tool is the \emph{root-section space} 
\[V_{x,y}:=\left\{\varphi\in(\mathfrak{h}^*)^{[x,y]}:\varphi(wt)-\varphi(w)\in\R\alpha_t\text{ for }x\leq w<wt\leq y\right\}/\sim\]
where $\varphi\sim\varphi'$ if they differ by a constant. This is the space considered in \cite[Section~4]{Patimo}, consisting of degree-two sections of the structure sheaf on the Bruhat moment graph. Dyer \cite{Dyer97} proved that $\dim V_{x,y}=-[q^{\ell(x,y)-1}]R_{x,y}(q)$. See also \cite[Remark~4.3]{Patimo} and \cite{BarkleyGaetzLam}.

For $x\leq w<wt\leq y$, define $\lambda_{w,wt}:V_{x,y}\rightarrow\R$ by 
\[\lambda_{w,wt}(\varphi)\cdot\alpha_t=\varphi(wt)-\varphi(w).\] 
For an atom $a\in \At(x,y)$, we write $\lambda_a:=\lambda_{x,a}$ for simplicity. Define 
\[C_{x,y}:=\mathrm{cone}\{\lambda_a:a\in\At(x,y)\}\subseteq V_{x,y}^*.\]

Throughout this section, fix $\eta\in\mathfrak{h}^*$ such that $\eta(\alpha_s^{\vee})>0$ for all $s\in S$. Subsequently fix $\theta:[x,y]\rightarrow\mathfrak{h}^*$ defined by $\theta(w)=-w^{-1}\eta$. Then for $x\leq w<wt\leq y$, we calculate that 
\[\theta(wt)-\theta(w)=-tw^{-1}\eta+w^{-1}\eta=(w^{-1}\eta)(\alpha_t^{\vee})\alpha_t=\eta(w\alpha_t^{\vee})\cdot\alpha_t.\]
Since $w<wt$, $w\alpha_t^{\vee}$ is a positive coroot and thus $\eta(w\alpha_t^{\vee})>0$. This means that $\theta\in V_{x,y}$ and moreover, $\lambda_{w,wt}(\theta)>0$ for all $w<wt$ in the interval $[x,y]$. In particular, $\lambda_a(\theta)>0$ for all $a\in\At(x,y)$, meaning that $C_{x,y}$ is a pointed cone. We can then take a vertex figure 
\[L_{x,y}:=\mathrm{conv}\{\lambda_a/\lambda_a(\theta):a\in\At(x,y)\}\subseteq V_{x,y}^{*}.\]
The remaining goal of this section is to show that for a lattice Bruhat interval $[x,y]$, the cone $C_{x,y}$, or equivalently, the polytope $L_{x,y}$, has the desired complete face lattice $[x,y]$.

\subsection{Basic properties}
We start with a rank-two calculation. See also \cite[Lemma~3.2]{Jahn-Stump}.
\begin{lemma}\label{lem:rank-2-positivity}
Let $[u,v]=\{u,a,b,v\}$ be a diamond. Then we have a relation 
\[\lambda_{u,b}=\kappa\lambda_{u,a}+\tau\lambda_{a,v},\quad\text{for }\kappa\in\R,\ \tau\in\R_{>0}.\]
In particular, for $\varphi\in V_{x,y}$, if $\lambda_{e_1}(\varphi)=\lambda_{e_2}(\varphi)=0$ for two adjacent edges $e_1$ and $e_2$ of this diamond, then $\lambda_e(\varphi)=0$ for every edge $e$ of this diamond.
\end{lemma}
\begin{proof}
The coefficients on either saturated chain of the diamond determine the rest. Hence there are scalars $\kappa,\tau\in\R$ such that $\lambda_{u,b}=\kappa\lambda_{u,a}+\tau\lambda_{a,v}.$

Let $\alpha,\beta,\gamma$ be the positive left coroot labels of $u\lessdot a$, $u\lessdot b$, and $a\lessdot v$, respectively. Evaluating this relation on $\theta_\zeta(w)=-w^{-1}\zeta$ gives $\beta=\kappa\alpha+\tau\gamma$. 

Since $\alpha$ and $\beta$ are independent, $\tau\neq0$. Choose $\zeta\in\mathfrak h^*$ with $\zeta(\alpha)=0$ and $\zeta(\gamma)=1$. Order positive coroots by $\zeta(\rho)/\eta(\rho)$, breaking ties lexicographically using fixed linear coordinates of $\beta/\eta(\beta)$. This is a reflection order \cite[Section~2.2]{Dyer-Hecke}. The chain $u\lessdot a\lessdot v$ is strictly increasing, so it is lexicographically first \cite[Proposition~4.3]{Dyer-Hecke}. Its first label is $0$, whereas the first label of the other chain is $\zeta(\beta)/\eta(\beta)=\tau/\eta(\beta)$. Therefore $\tau>0$.
\end{proof}

\begin{lemma}\label{lem:lattice-bruhat-dimension}
Let $[x,y]$ be a lattice Bruhat interval of rank $n$. Then $\dim V_{x,y}=n$. Moreover, for any maximal chain $x=x_0\lessdot x_1\lessdot\cdots\lessdot x_{n}=y$, the map $\oplus_{i=1}^n\lambda_{x_{i-1},x_{i}}:V_{x,y}\rightarrow\R^n$ is a linear isomorphism of vector spaces.
\end{lemma}
\begin{proof}
It is clear that this map is linear. By Dyer \cite{Dyer97} and Theorem~\ref{thm:lattice-characterization}, $\dim V_{x,y}=-[q^{\ell(x,y)-1}]R_{x,y}(q)=n$. It then suffices to show that this map is injective. Since this chain diamond generates, by Lemma~\ref{lem:rank-2-positivity}, for $\varphi\in V_{x,y}$, if $\lambda_{x_{i-1},x_{i}}(\varphi)=0$ for all $i\in[n]$, then indeed $\lambda_{w,wt}(\varphi)=0$ for all edges $(w,wt)$ of the Bruhat graph.
\end{proof}

Lemma~\ref{lem:lattice-bruhat-dimension} shows that for any maximal chain $x=x_0\lessdot x_1\lessdot\cdots\lessdot x_{n}=y$ in a lattice Bruhat interval $[x,y]$ of rank $n$, we may arbitrarily specify $\lambda_{x_{i-1},x_i}(\varphi)$ for all $i\in[n]$ to recover a unique $\varphi\in V_{x,y}$. We now specify some special elements in $V_{x,y}$, which will later serve as supporting vectors for facets of $C_{x,y}$.

\begin{definition}\label{def:fc}
For each coatom $c\in\Co(x,y)$, define $f^c\in V_{x,y}$ such that $\lambda_{x_{i-1},x_i}(f^c)=0$ for $i=1,\ldots,n-1$ and $\lambda_{c,y}(f^c)=1$, where $x=x_0\lessdot x_1\lessdot\cdots\lessdot x_{n-1}=c\lessdot y$ is a maximal chain. This definition does not depend on the choice of maximal chain from $x$ to $c$.
\end{definition}
\begin{ex}
Consider the lattice Bruhat interval $[1324,4231]$ in $S_4$. Choose a coatom $c=3241$, and we compute $f^c$ as in Figure~\ref{fig:fc}. Note that among the atoms, $1342,2314,3124$ are below $c$ while $1423\nleq c$ and we have $\lambda_{1423}(f^c)=1$.
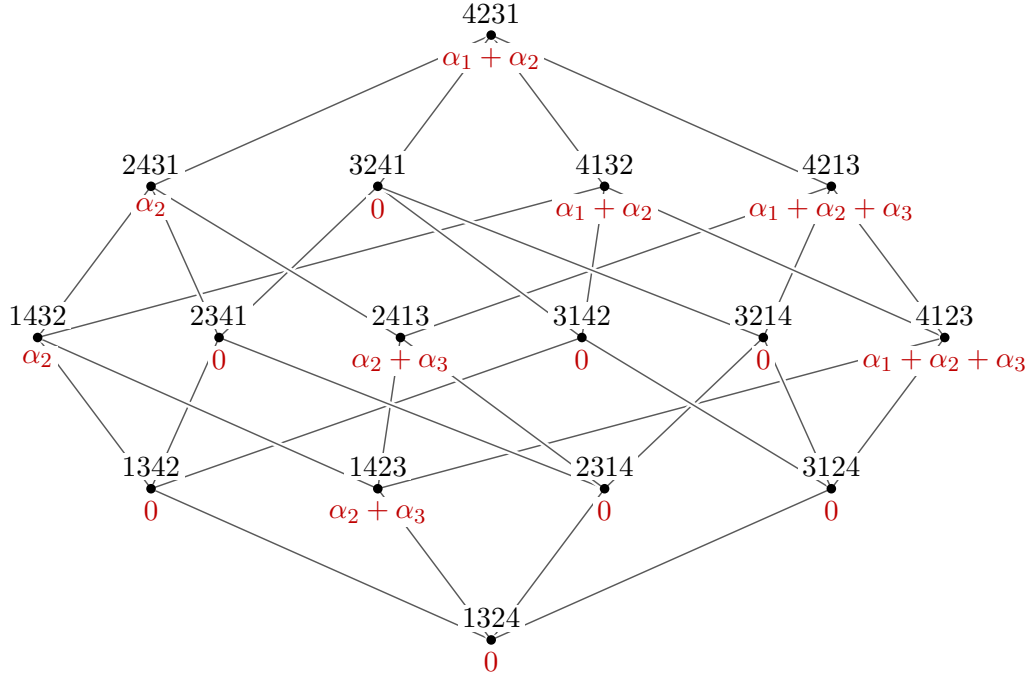
\begin{figure}[h!]
\centering
\begin{tikzpicture}[
    x=1cm,
    y=1cm,
    edge/.style={
        draw=black!65,
        line width=.55pt,
        preaction={draw=white,line width=2.4pt}
    },
    vertex/.style={
        circle,
        fill=black,
        inner sep=0pt,
        minimum size=3.6pt
    },
    permutation/.style={
        font=\small,
        fill=white,
        inner sep=1.5pt
    },
    value/.style={
        font=\small,
        text=red!75!black,
        fill=white,
        inner sep=1.5pt
    }
]


\coordinate (v1324) at (0,0);

\coordinate (v1342) at (-4.5,2);
\coordinate (v1423) at (-1.5,2);
\coordinate (v2314) at ( 1.5,2);
\coordinate (v3124) at ( 4.5,2);

\coordinate (v1432) at (-6.0,4);
\coordinate (v2341) at (-3.6,4);
\coordinate (v2413) at (-1.2,4);
\coordinate (v3142) at ( 1.2,4);
\coordinate (v3214) at ( 3.6,4);
\coordinate (v4123) at ( 6.0,4);

\coordinate (v2431) at (-4.5,6);
\coordinate (v3241) at (-1.5,6);
\coordinate (v4132) at ( 1.5,6);
\coordinate (v4213) at ( 4.5,6);

\coordinate (v4231) at (0,8);

\foreach \u/\v in {
    1324/1342,1324/1423,1324/2314,1324/3124,
    1342/1432,1342/2341,1342/3142,
    1423/1432,1423/2413,1423/4123,
    2314/2341,2314/2413,2314/3214,
    3124/3142,3124/3214,3124/4123,
    1432/2431,1432/4132,
    2341/2431,2341/3241,
    2413/2431,2413/4213,
    3142/3241,3142/4132,
    3214/3241,3214/4213,
    4123/4132,4123/4213,
    2431/4231,3241/4231,4132/4231,4213/4231%
}{
    \draw[edge] (v\u) -- (v\v);
}

\foreach \w/\f in {
    1324/{0},
    1342/{0},
    1423/{\alpha_2+\alpha_3},
    2314/{0},
    3124/{0},
    1432/{\alpha_2},
    2341/{0},
    2413/{\alpha_2+\alpha_3},
    3142/{0},
    3214/{0},
    4123/{\alpha_1+\alpha_2+\alpha_3},
    2431/{\alpha_2},
    3241/{0},
    4132/{\alpha_1+\alpha_2},
    4213/{\alpha_1+\alpha_2+\alpha_3},
    4231/{\alpha_1+\alpha_2}%
}{
    \node[vertex] at (v\w) {};
    \node[permutation,anchor=south,yshift=3pt]
        at (v\w) {$\w$};
    \node[value,anchor=north,yshift=-3pt]
        at (v\w) {$\f$};
}

\end{tikzpicture}
\caption{A representative of $f^c$, normalized to vanish at $1324$}
\label{fig:fc}
\end{figure}
\end{ex}

\subsection{The main positivity lemma}
The following positivity lemma is the main step.
\begin{lemma}\label{lem:positivity-main}
Let $[x,y]$ be a lattice Bruhat interval and define $f^c$ for $c\in\Co(x,y)$ as in Definition~\ref{def:fc}. For an atom $a\in\At(x,y)$, $\lambda_a(f^c)>0$ if $a\nleq c$ and $\lambda_a(f^c)=0$ if $a\leq c$.
\end{lemma}

We first develop some tools. Define the \emph{atom graph} $A(x,y)$, which has vertex set $\At(x,y)$ and an edge $a\sim b$ if there exists $z\in [x,y]$ such that $a\lessdot z$ and $b\lessdot z$. This graph is connected, since all maximal chains are connected by diamond flips (see for example \cite{Blanco-flip-posets}).

\begin{lemma}\label{lem:short-ordinary-sections}
Let $[x,y]$ be a lattice Bruhat interval with $1\leq\ell(x,y)\leq3$. Every $\varphi\in V_{x,y}$ is represented by $w\mapsto-w^{-1}\zeta$ for some $\zeta\in\mathfrak h^*$.
\end{lemma}
\begin{proof}
Put $n=\ell(x,y)$ and choose a maximal chain $x=x_0\lessdot x_1\lessdot\cdots\lessdot x_n=y$, and let $\beta_i$ be the positive left coroot label of $x_{i-1}\lessdot x_i$. We claim that $\beta_1,\ldots,\beta_n$ are linearly independent. There is nothing to do if $n\leq2$ so assume $n=3$. If $\beta_1,\beta_2,\beta_3$ are linearly dependent, then their span $V$ has dimension $2$ and their corresponding reflections generate a dihedral subgroup. All the left coroot labels in $[x,y]$ are in the span $V$ since the maximal chains are connected by diamond flips. The whole interval $[x,y]$ must be isomorphic to a subinterval of the Bruhat order of a dihedral group (see for example \cite[Theorem~1.4]{Dyer-Bruhat-graph}), which has to be a $2$-crown since $\ell(x,y)=3$, contradicting $[x,y]$ being a lattice.

Since $\beta_1,\ldots,\beta_n$ are linearly independent, we can choose $\zeta\in\mathfrak{h}^*$ so that $\zeta(\beta_i)=\lambda_{x_{i-1},x_i}(\varphi)$ for all $i$. Define $\theta_{\zeta}\in V_{x,y}$ by $\theta_{\zeta}(w)=-w^{-1}\zeta$. As in the calculation for $\theta$, we have $\theta_{\zeta}(wt)-\theta_\zeta(w)=\zeta(w\alpha_t^{\vee})\alpha_t$. Therefore, $\theta_{\zeta}$ and $\varphi$ agree on this maximal chain, and thus agree everywhere.
\end{proof}

\begin{lemma}\label{lem:sublevel-connectivity}
Let $[x,y]$ be a lattice Bruhat interval. For every $\varphi\in V_{x,y}$ and $t\in\R$, the subgraphs of $A(x,y)$ induced by
\[
\{a\in\At(x,y):\lambda_a(\varphi)/\lambda_a(\theta)<t\}\ \text{and}\ \{a\in\At(x,y):\lambda_a(\varphi)/\lambda_a(\theta)\leq t\}
\]
are connected whenever nonempty.
\end{lemma}
\begin{proof}
Call $\varphi$ \emph{generic} if $\lambda_e(\varphi)/\lambda_e(\theta)\neq\lambda_{e'}(\varphi)/\lambda_{e'}(\theta)$ whenever $\lambda_e/\lambda_e(\theta)$ and $\lambda_{e'}/\lambda_{e'}(\theta)$ are distinct functionals. Note that when the left coroot labels $\beta_e$ and $\beta_{e'}$ are distinct, then they indeed provide distinct functionals. Such generic sections form a dense open subset of $V_{x,y}$, since only finitely many proper hyperplanes are excluded. We first assume that $\varphi$ is generic, and label each cover edge $e$ by $\lambda_e(\varphi)/\lambda_e(\theta)\in\R$, reading chains from bottom to top. We now show that each subinterval of rank at most $3$ has a unique strictly increasing maximal chain, whose label sequence is lexicographically minimal among all maximal chains in that subinterval.

For a subinterval of rank at most $3$, by Lemma~\ref{lem:short-ordinary-sections}, $\varphi$ is represented by $w\mapsto -w^{-1}\zeta$ for some $\zeta\in\mathfrak{h}^*$ locally. Consequently, the label of an edge $e$ in this subinterval is $\zeta(\beta_e)/\eta(\beta_e)$. Order the positive coroots by these ratios, breaking ties lexicographically using fixed linear coordinates of $\beta/\eta(\beta)$. Since normalization by $\eta$ turns positive linear combinations into convex combinations, this is a reflection order by \cite[Section~2.2]{Dyer-Hecke}. Genericity ensures that its comparisons on distinct coroots occurring in the subinterval agree with the numerical labels. The assertion therefore follows from \cite[Proposition~4.3]{Dyer-Hecke},
using inversion to pass to left labels. 

We then analyze maximal chains in $[x,y]$ with the numerical labels. A \emph{descent flip} replaces a decreasing two-edge segment of a maximal chain by the other side of its diamond. Such a flip strictly decreases the label sequence lexicographically. Since there are finitely many maximal chains, every sequence of descent flips terminates. Successive edges have distinct left coroot labels and hence distinct numerical labels, so a terminal chain is strictly increasing. We claim that the terminal chain is independent of all choices. Suppose two descent flips are available at the same chain. If their two-edge segments do not overlap in an edge, the flips commute. Otherwise, both segments lie in a rank-three subinterval. After either first flip, further descent flips within that subinterval terminate at its unique increasing maximal chain, while leaving the rest unchanged. Thus the two resulting chains have a common descendant. Since descent flips terminate and are locally confluent, Newman's lemma \cite[Theorem~3]{Newman1942} (also called Rome's lemma) implies that every maximal chain has a unique terminal chain. Finally, maximal chains are connected by diamond flips, and each diamond flip is a descent flip in one direction. Therefore all maximal chains have the same terminal chain.

We can now prove the connectivity in the original lemma statement. Let $a_0$ be the first atom of the common terminal chain that is strictly increasing. Starting from a maximal chain through any atom $a$, perform descent flips and record the atoms inside. Each flip either leaves the atom unchanged or replaces it by an adjacent atom with strictly smaller height. This gives a path from $a$ to $a_0$ in the atom graph $A(x,y)$ along which $\lambda_a(\varphi)/\lambda_a(\theta)$ strictly decreases. Thus, each of the nonempty sets in the lemma statement contains $a_0$, and each of its vertices is connected to $a_0$ within the set.

Finally, let $\varphi$ be arbitrary and consider either selected set of atoms. The empty and full sets cause no difficulty. Otherwise, every selected height is strictly smaller than every unselected height. Choose a threshold strictly between these two finite sets of heights, and perturb $\varphi$ to a sufficiently close generic section.
\end{proof}

We are now ready to establish the main technical lemma.

\begin{proof}[Proof of Lemma~\ref{lem:positivity-main}]
By construction, $\lambda_{w,wt}(f^c)=0$ if $w<wt\leq c$. In particular, $\lambda_a(f^c)=0$ for $a\leq c$. The interesting part is showing $\lambda_a(f^c)>0$ for $a\nleq c$.

We induct on $n=\ell(x,y)$. The base case $n=2$ follows from Lemma~\ref{lem:rank-2-positivity}. 

Consider an edge $a\sim b$ of $A(x,y)$ with $a\nleq c$ and $b\leq c$, and put $z=a\vee b$. Then $[x,z]=\{x,a,b,z\}$ is a diamond.  Since $z\in\At(b,y)$ and $z\nleq c$, the induction hypothesis applied to the restriction of $f^c$ to $[b,y]$ gives $\lambda_{b,z}(f^c)>0$. Lemma~\ref{lem:rank-2-positivity} therefore implies
\[
\lambda_a(f^c)=\kappa\lambda_b(f^c)+\tau\lambda_{b,z}(f^c)=\tau\lambda_{b,z}(f^c)>0,\qquad \text{for }\tau>0.
\]
Thus every edge leaving the atoms below $c$ has a strictly positive outside endpoint.

Now consider the set
\[
S_c=\{a\in\At(x,y):\lambda_a(f^c)\leq0\}.
\]
Since normalizing by $\lambda_a(\theta)>0$ does not change the sign, Lemma~\ref{lem:sublevel-connectivity} says that its induced subgraph is connected. It contains the nonempty set of atoms below $c$. If it also contained an atom outside $c$, a path within $S_c$ starting below $c$ would have a first edge leaving $c$. Its outside endpoint has strictly positive evaluation by the preceding paragraph, a contradiction. Hence $S_c=\{b\in\At(x,y):b\leq c\}$, proving the desired strict positivity.
\end{proof}

\subsection{The face structure}
In our construction, the pointed cone $C_{x,y}$ lives in $V_{x,y}^*$ so its supporting vectors live in $V_{x,y}$.  Recall from Section~\ref{sub:normal-fan} that a supporting vector $f$ of a pointed cone $C$ needs to satisfy $f(v)\geq0$ for all $v\in C$, hence the importance of Lemma~\ref{lem:positivity-main}.
\begin{proof}[Proof of Theorem~\ref{thm:main}]
For a lattice Bruhat interval $[x,y]$, we show that the complete face lattice of $C_{x,y}$, which is also the complete face lattice of $L_{x,y}$, is isomorphic to $[x,y]$.

Note that $[x,y]$ is atomic, i.e. every element is the join of atoms below it. This is because every $z\in[x,y]$ with $\ell(x,z)\geq2$ has at least two elements $z_1$ and $z_2$ covered by $z$ so that $z=z_1\vee z_2$ and then we can apply induction to $z_1$ and $z_2$. Similarly, $[x,y]$ is coatomic, meaning that every element is the meet of coatoms above it.

For $z\in[x,y]$, put $K_z:=\mathrm{cone}\{\lambda_a:a\in\At(x,y),a\leq z\}$, so that $K_y=C_{x,y}$. We show 
\[K_z=C_{x,y}\cap\bigcap_{c\in\Co(x,y),c\geq z}\ker(f^c).\]
The $\subseteq$ direction follows from $\lambda_a\in\ker(f^c)$ (Lemma~\ref{lem:positivity-main}) for $a\le z\leq c$ where $a\in\At(x,y)$ and $c\in\Co(x,y)$. For the $\supseteq$ direction, take an element $\sum_{a\in\At(x,y)}c_a\lambda_a$, $c_a\in\R_{\geq0}$ on the right hand side. In order for it to be in $\ker(f^c)$ for a fixed coatom $c\geq z$, by Lemma~\ref{lem:positivity-main}, we must have that $c_b=0$ if $b\nleq c$. Any atom $a$ with $c_a>0$ must lie below every coatom $c\geq z$, and hence below $z=\bigwedge_{c\in\Co(x,y),c\geq z}c$ by coatomicity. 

Thus the $K_z$'s are faces of $C_{x,y}$. Clearly for $z\leq z'$, $K_z\subseteq K_{z'}$. Also, if $z\lessdot z'$, atomicity implies that $z'=z\vee a$ for some atom $a\nleq z$. Coatomicity says that there exists some coatom $c\ngeq z'$ such that $z=z'\wedge c$. Then $a\nleq c$ and $\lambda_a\notin\ker(f^c)$ which means that $\lambda_a\in K_{z'}\setminus K_z$, $K_z\subsetneq K_{z'}$. At the same time, $K_{z}=K_{z'}\cap\ker(f^c)$ so $\dim(K_{z'})-\dim(K_z)\geq1$. Together with $V_{x,y}^*$ having dimension $\ell(x,y)$, we deduce that $\dim(K_z)=\ell(x,z)$ for all $z\in[x,y]$. We calculate that 
\[K_{z_1}\cap K_{z_2}=C_{x,y}\cap\bigcap_{c\in\Co(x,y),\ c\geq z_1\text{ or }c\geq z_2}\ker(f^c)=K_{z}\]
where $z=\bigwedge_{c\in\Co(x,y),\ c\geq z_1\text{ or }c\geq z_2}c=\left(\bigwedge_{c\in\Co(x,y),c\geq z_1}c\right)\wedge\left(\bigwedge_{c\in\Co(x,y),c\geq z_2}c\right)=z_1\wedge z_2$. In particular, the equation $K_{z_1}\cap K_{z_2}=K_{z_1\wedge z_2}$ implies that all the $K_z$'s are distinct faces, by dimensional considerations. Then if $K_{z}\subseteq K_{z'}$, $K_z=K_z\cap K_{z'}=K_{z\wedge z'}$, $z=z\wedge z'$ so $z\leq z'$. Thus, $z\mapsto K_z$ is an order embedding from $[x,y]$ to $\mathcal{F}(C_{x,y})$.

Finally, there are no additional faces. It is a well-known fact that all maximal chains in the complete face lattice of a polytope are connected by diamond flips. See for example \cite{Bayer-Barycentric}, or use Balinski's theorem \cite[Theorem~3.14]{Ziegler} to establish that the facet-adjacency graph is connected and apply induction. Starting from any maximal chain $K_{z_0}\subsetneq K_{z_1}\subsetneq\cdots\subsetneq K_{z_n}$, all possible diamond flips can involve only some other known face $K_{z'}$, as all rank $2$ Bruhat intervals are diamonds. Therefore, all faces of $C_{x,y}$ are known.
\end{proof}

\section{Toric realizations and Richardson varieties}\label{sec:geometry}
Throughout this section, $G$ is a connected complex semisimple algebraic group, $B$ and $B^{-}$ are opposite Borel subgroups, $T:=B\cap B^{-}$ is a maximal torus, and $W$ is the Weyl group. We write $w_0$ for the longest element of $W$ and $\dot wB$ for the $T$-fixed point of the flag variety $G/B$ indexed by $w\in W$.  Let 
\[
 C_w=B\dot wB/B,\qquad C^w=B^{-}\dot wB/B, \qquad X_w=\overline{C_w},\qquad X^w=\overline{C^w},
\]
be the Schubert cell, opposite Schubert cell, Schubert variety, and opposite Schubert variety, respectively. For $x\leq y$ in the Bruhat order, recall we have the open and closed Richardson varieties
\[
 R_{x,y}=C^x\cap C_y, \qquad \overline{R}_{x,y}=X^x\cap X_y.
\]
They have complex dimension $\ell(x,y)$, and 
\begin{equation*}
 \overline{R}_{x,y} =\bigsqcup_{x\leq a\leq b\leq y}R_{a,b}.
\end{equation*}

We use rational coefficients for intersection cohomology. 

\subsection{The toric Richardson construction}
We begin with the geometric result due to Gorsky--Kim--Sherman-Bennett in type A \cite[Theorems~1.1 and~1.2]{GorskyKimShermanBennett} and Bossinger--Gorsky--Simental in general type \cite[Theorems~6.2 and~6.6]{BossingerGorskySimental2026}.
\begin{proposition}\label{prop:toric-Richardson}
For $x\leq y$ in a finite Weyl group, the following are equivalent:
\begin{enumerate}
\item $[x,y]$ is a lattice;
\item $R_{x,y}$ is isomorphic to an algebraic torus;
\item $\overline{R}_{x,y}$ is a toric variety.
\end{enumerate}
When these conditions hold, the torus structure on $R_{x,y}$ extends to $\overline{R}_{x,y}$, and the torus orbits in $\overline{R}_{x,y}$ are the Richardson strata $R_{a,b}$ for $x\le a\le b\le y$. Consequently, the complete face lattice of  the moment polytope $Q_{x,y}$ of $\overline{R}_{x,y}$ is given by 
\begin{equation*}
 \mathcal F(Q_{x,y})\cong\operatorname{Int}([x,y]) :=\{\varnothing\}\sqcup\{[a,b]:x\leq a\leq b\leq y\},
\end{equation*}
ordered by containment. The face indexed by $[a,b]$ has dimension $\ell(a,b)$.
\end{proposition}

Here toric means toric for some effective algebraic torus of the appropriate dimension. It does not require the dense orbit to be an orbit of the maximal torus $T$. In particular,  $\ell(x,y)$ can be larger than the rank of $G$. 

Assume from now on that $x<y$ and $[x,y]$ is a lattice, and let $n=\ell(x,y)$. Let $q_x$ be the vertex of $Q_{x,y}$ indexed by the degenerate interval $[x,x]$. Choose a rational supporting vector $\lambda$ of $q_x$ and a sufficiently small rational $\varepsilon>0$. Let
\begin{equation*}
 \Pi_{x,y}:=\mathrm{VF}_{q_x}(Q_{x,y})
\end{equation*}
be the vertex figure of $Q_{x,y}$ at $q_x$. Geometrically, $\Pi_{x, y}$ is the new facet created by cutting off the vertex $q_x$. We use $\Pi_{x,y}$ for this geometric realization to distinguish it from the root-section realization $L_{x,y}$ of Section~\ref{sec:polytope}; their face lattices agree, but we do not require them to be affinely isomorphic.

\begin{corollary}\label{cor:geometric-realization}
The complete face lattice of the rational $(n-1)$-polytope $\Pi_{x,y}$ is $[x,y]$. Furthermore, there is a projective toric $(n-1)$-fold $X_{x,y}$ with moment polytope $\Pi_{x,y}$ and
\[
X_{x,y}=\bigsqcup_{x<w\leq y}O_w,\qquad O_w\subseteq\overline{O_z}\ \Longleftrightarrow\ w\leq z, \qquad \dim O_w=\ell(x,w)-1.
\]
Thus $[x,y]$ is the poset of torus-orbit closures in $X_{x,y}$, with the empty set adjoined. For $n\geq2$, $X_{x,y}$ is the new invariant divisor in a projective toric modification of $\overline{R}_{x,y}$.
\end{corollary}

\begin{proof}
A face $[a,b]$ of $Q_{x,y}$ satisfies $x\le a\le b\le y$. It contains $q_x=[x, x]$ precisely when $a\leq x\leq b$, hence precisely when $a=x$. The interval of faces containing $q_x$ is therefore
\[
 \{[x,w]:x\leq w\leq y\}\cong[x,y].
\]
Taking the vertex figure identifies $q_x$ with the empty face of $\Pi_{x,y}$, and the face $[x,w]$ of $Q_{x,y}$, for $w>x$, with a face of $\Pi_{x,y}$ of dimension $\ell(x,w)-1$. This proves the first assertion; compare \cite[Corollary~6.8]{BossingerGorskySimental2026}.

Choose a positive integral dilation making $\Pi_{x,y}$ a lattice polytope. Its normal fan defines a projective toric variety, and the assertions about its orbits follow from the orbit-cone correspondence \cite[Theorem 3.2.6]{CLS2011}. For $n\geq2$, truncate $Q_{x,y}$ at $q_x$ by setting
\[
 Q'_{x,y}=Q_{x,y}\cap \{\lambda\leq\lambda(q_x)-\varepsilon\}.
\]
The normal fan of $Q'_{x,y}$ refines the normal fan of $Q_{x,y}$. The resulting projective equivariant toric modification 
$$X_{Q'_{x,y}}\longrightarrow X_{Q_{x,y}}=\overline{R}_{x,y}$$
creates a new torus-invariant divisor $X_{x,y}$ whose moment polytope is exactly $\Pi_{x,y}$ (see the constructions of projective toric varieties and fan refinements in \cite{CLS2011}). For $n=1$, the variety $X_{x,y}$ is a point; no exceptional divisor construction is needed.
\end{proof}

\subsection{Local and global intersection cohomology of toric Richardson varieties}
Let $\mathcal L$ be a face lattice of a polytope with rank function $\rho$. Let $f_{A,B}(q)$ denote the right KLS polynomials associated with the Eulerian kernel $\epsilon$, which are characterized by
\begin{equation*}
 \begin{split}
 f_{A,A}(q)&=1,\\
 \deg f_{A,B}(q)&<\frac{\rho(B)-\rho(A)}2 \text{ for } A<B,\\
 q^{\rho(B)-\rho(A)}f_{A,B}(q^{-1})&=\sum_{A\leq D\leq B}(q-1)^{\rho(D)-\rho(A)}f_{D,B}(q),
 \end{split}
\end{equation*}
where the last identity is the expansion of the defining property of the right KLS function from the convolution product of the incidence algebra in Section~\ref{sub:KLS}.

Let $[x,y]$ be a lattice Bruhat interval in a finite Weyl group. By Proposition~\ref{prop:toric-Richardson}, we have the face lattice $\mathcal{F}(Q_{x,y})\cong\operatorname{Int}([x,y])$. We take the rank convention
\begin{equation*}
 \rho(\varnothing)=0,
 \qquad \rho([a,b])=\ell(a,b)+1.
\end{equation*}
Thus a vertex of $Q_{x,y}$ has rank one and the rank difference between two nonempty interval faces of $Q_{x,y}$ is $\ell(c,d)-\ell(a,b)$ when $[a,b]\subseteq[c,d]$. Then we can compute the Poincar\'e
polynomials of the local and global intersection cohomology of the toric closed Richardson variety by the following theorem.

\begin{theorem}\label{thm:Richardson-IH}
Let $[x,y]$ be a lattice Bruhat interval in a finite Weyl group. For $x\leq c\leq a\leq b\leq d\leq y$ and $p\in R_{a,b}$, the odd local intersection cohomology of $\overline{R}_{c,d}$ at $p$ vanishes, and
\begin{equation*}
 \sum_{i\geq0}\dim\mathrm{IH}^{2i}_p(\overline{R}_{c,d})q^i=f_{[a,b],[c,d]}(q)=P_{w_0a,w_0c}(q)P_{b,d}(q).
\end{equation*}
The odd global intersection cohomology of $\overline{R}_{x,y}$ also vanishes, and the corresponding Poincar\'e polynomial can be computed by the toric $h$-polynomial
\begin{align*}
 H_{x,y}(q)&:={\sum}_{i\geq0}\dim\mathrm{IH}^{2i}(\overline{R}_{x,y})q^i\\
 &=\sum_{x\leq a\leq b\leq y}(q-1)^{\ell(a,b)}f_{[a,b],[x,y]}(q)\\
 &=\sum_{x\leq w\leq y}q^{\ell(x,w)}P_{w_0w,w_0x}(q^{-1})P_{w,y}(q).
\end{align*}
\end{theorem}

\begin{proof}
The local intersection cohomology of a toric orbit closure in a projective toric variety is computed by the right KLS polynomial of its orbit-closure interval \cite{DenefLoeser1991,Fieseler1991} (see also \cite[Corollary~4.20]{Proudfoot2018}). Applying Proposition~\ref{prop:toric-Richardson} gives the vanishing in odd degrees and the first equality in the local formula. The local product theorem for Richardson varieties \cite[Section~3.1, p.~397]{KnutsonWooYong2013} gives the second equality.

For a projective toric variety with moment polytope $Q$, Stanley's toric $h$-polynomial $h(Q^*,q)$ computes the Poincar\'e polynomial of its global intersection cohomology, in our convention (see \cite{DenefLoeser1991,Fieseler1991}). Applying Proposition~\ref{prop:toric-Richardson} gives the vanishing in odd degrees and the second equality in the global situation. The Poincar\'e polynomial of the global intersection cohomology of any closed Richardson variety $\overline{R}_{x, y}$ (not necessarily toric) can be computed using the classical KL polynomials (see the Richardson $Z$-polynomial formula of \cite[Theorem~4.3 and Remark~4.4]{Proudfoot2018}). This gives the last equality in the global formula.
\end{proof}

\subsection{Rational smoothness of toric Richardson varieties}
Recall that an irreducible $n$-dimensional complex algebraic variety $X$ is \emph{rationally smooth} if for all $x \in X$,
\[
 H^i(X,X\setminus\{x\};\mathbb Q)\cong\begin{cases}\mathbb Q,&i=2n,
 \\0,&i\neq2n.\end{cases}
\]
If $X$ is rationally smooth, then its rational ordinary cohomology is isomorphic to its rational intersection cohomology. If furthermore $X$ is projective, then its rational ordinary cohomology satisfies Poincar\'e duality. 

When the moment polytope $Q_{x,y}$ is simple, the toric closed Richardson variety $\overline{R}_{x, y}$ is rationally smooth. The following theorem gives equivalent criteria for the rational smoothness of the toric Richardson variety $\overline{R}_{x,y}$.

\begin{theorem}\label{thm:rational-smoothness}
Let $[x,y]$ be a lattice Bruhat interval in a finite Weyl group, and let $n=\ell(x,y)$. The following are equivalent:
\begin{enumerate}
\item $\overline{R}_{x,y}$ is rationally smooth;
\item $Q_{x,y}$ is simple;
\item $f_{[x,x],[x,y]}(q)=1$;
\item $f_{[y,y],[x,y]}(q)=1$;
\item $P_{x,y}(q)=1$;
\item $Q_{x,y}$ is combinatorially equivalent to the $n$-cube;
\item $H_{x,y}(q)=(1+q)^n$;
\item $\Pi_{x,y}$ is a simplex;
\item $[x,y]$ is isomorphic to the Boolean lattice $B_n$.
\end{enumerate}
When these conditions hold, $H_{x,y}(q)$ also equals $\sum_i\dim H^{2i}(\overline{R}_{x,y};\mathbb Q)q^i$, and the odd ordinary cohomology vanishes.
\end{theorem}

\begin{proof}
A toric variety is rationally smooth precisely when its fan is simplicial \cite[Theorem~11.4.8]{CLS2011}. Since the fan of $\overline{R}_{x,y}$ is the normal fan of $Q_{x,y}$, (1) and (2) are equivalent. In particular, (1), equivalently, (2), implies (3) and (4).

The local intersection cohomology polynomial at a toric fixed point is the toric $g$-polynomial of a section of the corresponding full-dimensional cone. For a $d$-polytope $P$, the linear coefficient of the toric $g$-polynomial is $f_0(P)-(d+1)$ (see, for example, \cite[(5), p.~1153]{Braden2006}). Since a toric $g$-polynomials cannot have internal zeros, it follows that this local polynomial is $1$ if and only if the cone is simplicial. Hence (3) says that $q_x$ is a simple vertex. 
Its vertex figure $\Pi_{x,y}$ is then a simplex, so (3) implies (8). The same argument at $q_y$, 
whose vertex figure has complete face lattice anti-isomorphic to $[x,y]$, shows that (4) implies (9).

By Corollary~\ref{cor:geometric-realization}, (8) and (9) are equivalent. 

If $[x,y]\cong B_n$, then its interval poset is the face lattice of the $n$-cube: a pair of subsets $A\subseteq B\subseteq[n]$ specifies the cube face whose coordinates in $A$ are $1$, those outside $B$ are $0$, and those in $B\setminus A$ are free. This combined with Proposition~\ref{prop:toric-Richardson} proves (9) implies (6), which implies (2). Conversely, a vertex figure of a cube is a simplex, so (6) implies (8).

By Theorem~\ref{thm:Richardson-IH}, $f_{[x,x],[x,y]}(q)=P_{x,y}(q)$. Thus (3) and (5) are equivalent. This establishes all the equivalences except (7).

Note that Stanley's toric $h$-polynomial $h(P,q)$ depends only on the face lattice and by the proof of Theorem~\ref{thm:Richardson-IH}, we have $H_{x,y}(q)=h(Q_{x,y}^*,q)$ in our convention, where $Q_{x,y}^*$ is the polar dual of $Q_{x,y}$. Hence $h(C_n^*,q)=(1+q)^n$ for an $n$-cube shows that (6) implies (7). 

Conversely, suppose (7) holds and $n\geq2$. For a projective toric $n$-fold with moment polytope $Q$, the coefficient of $q$ in its intersection cohomology Poincar\'e polynomial is $f_{n-1}(Q)-n$, where $f_{n-1}(Q)$ is the number of facets of $Q$. Equivalently, this is the linear coefficient of the toric $h$-polynomial  of the polar polytope (see \cite{Stanley1987}). 
By Proposition~\ref{prop:toric-Richardson}, the facets in $\operatorname{Int}([x,y])$ are $[x,c]$ and $[a,y]$, with $c\in\Co(x,y)$ and $a\in\At(x,y)$. Thus
\[
 \#\At(x,y)+\#\Co(x,y)=2n.
\]
These are the numbers of vertices and facets, respectively, of the $(n-1)$-polytope $\Pi_{x,y}$. 
Each is at least $n$. Consequently $\Pi_{x,y}$ has exactly $n$ vertices, so it is a simplex. The case $n=1$ is immediate. 

Finally, rational smoothness identifies ordinary and intersection cohomology, which proves the last assertion.
\end{proof}

For Richardson varieties in type $A$ admitting a dense orbit of the ambient maximal torus, the equivalence of conditions (2), (6), and (9) was shown in \cite[Proposition~5.6 and Theorem~5.7]{LeeMasudaPark2021}. Theorem~\ref{thm:rational-smoothness} extends these equivalences to lattice Bruhat intervals in arbitrary finite Weyl groups, allowing torus actions beyond the ambient maximal torus, and relates them to Kazhdan--Lusztig polynomials and intersection cohomology.

An infinite family of rationally smooth toric $\overline{R}_{x,y}$ comes from the ``large Bruhat hypercubes'' recently found by Ellenberg--Libedinsky--Plaza--Simental--Williamson \cite{EllenbergETAL}. For each $n\geq2$, they constructed a family of permutations $x_n,\leq y_n \in S_{2^n}$ and showed that the Bruhat interval $\left[x_n, y_n\right]$ is Boolean, with $\ell(x_n,y_n)=n 2^{n-1}$. Thus $\overline{R}_{x_n, y_n}$ is toric for some torus $T_n$ and its moment polytope is combinatorially a cube.

\begin{remark}\label{rmk:convention}
One needs to distinguish the normal-fan convention for a moment polytope from the face-fan convention (see also \cite[Remark 4.21]{Proudfoot2018}). Our convention is that the toric $g$-vector and the toric $h$-vector extend the simplicial $g$-vector and the simplicial $h$-vector. For example, if $C_n$ is an $n$-cube, then $h(C_n^*,q)=(1+q)^n$, which computes the intersection cohomology Poincar\'e polynomial of the toric variety associated with the normal fan of $C_n$.
\end{remark}

\section{Consequences for Kazhdan--Lusztig polynomials} \label{sec:KL}
Recall the language of Kazhdan--Lusztig--Stanley theory in Section~\ref{sub:KLS}. When $[x,y]$ is a lattice, the classical Kazhdan--Lusztig polynomial $P_{x,y}(q)$ is the right KLS polynomial associated with the Eulerian kernel on $[x,y]$, since every subinterval of $[x,y]$ is a lattice and the classical $R$-polynomial kernel equals the Eulerian kernel $\epsilon$ (see Theorem~\ref{thm:lattice-characterization}). By the geometric interpretation of the toric $g$-polynomial for rational polytopes \cite{DenefLoeser1991,Fieseler1991}, we have a new interpretation for the KL polynomial when the Bruhat interval $[x,y]$ in a finite Weyl group is a lattice.

\begin{corollary}\label{cor:local-toric-KL}
Let $[x,y]$ be a lattice Bruhat interval in a finite Weyl group. For the projective toric variety $X_{x,y}$ in Corollary~\ref{cor:geometric-realization}, one has
\[
 P_{w,y}(q)=\sum_{i\geq0}\dim\mathrm{IH}^{2i}_p(X_{x,y})q^i \qquad(x<w\leq y,\ p\in O_w).
\]
In particular, all odd local intersection cohomology groups of $X_{x,y}$ vanish.
\end{corollary}
The above restriction $w>x$ is essential: the bottom element $x$ labels the empty set, not a torus orbit of $X_{x,y}$. On the other hand, the polynomial $P_{x,y}(q)$ computes the local intersection cohomology Poincar\'e polynomial of the affine toric variety associated with the dual $C(\Pi_{x,y})^{\vee}$ of the cone over $\Pi_{x,y}$ at the unique fixed point in our convention, cf. the Remark at the end of Section~\ref{sec:geometry} (see also \cite[Example 3.2]{BradenProudfoot2026}).

We now return to arbitrary Coxeter groups. First we specify the toric polynomials explicitly.  For a $d$-dimensional polytope $P$, define recursively
\begin{equation*}
 \begin{split}
 h(P,q)&=\sum_{F\subsetneq P}g(F,q)(q-1)^{d-1-\dim F},\\
g(P,q)&=\operatorname{trunc}_{\leq\lfloor d/2\rfloor}\bigl((1-q)h(P,q)\bigr),
 \end{split}
\end{equation*}
where the sum is over all proper faces $F$ of $P$, including the empty face. The induction starts with $g(\varnothing,q)=h(\varnothing,q)=1$. Here truncation means retaining only powers at most the indicated bound. This is Stanley's convention \cite{Stanley1987} (see also \cite[(4), p.~1152]{Braden2006}). 

On the complete face lattice of $P$, by expanding the defining property of the right KLS function from the convolution product of the incidence algebra in Section~\ref{sub:KLS}, we have
\begin{equation*}
 f_{\varnothing,P}(q)=g(P^*;q),
\end{equation*}
where $f_{\varnothing,P}(q)$ is the right KLS polynomial associated with the Eulerian kernel $\epsilon$ and $P^*$ denotes the polar dual of $P$.

We have the following classical observation.
\begin{proposition}\label{prop:kernel-coincidence}
Let $L=[x,y]$ be a lattice Bruhat interval in a Coxeter group. For every $x\leq a\leq b\leq y$, the right KLS polynomial $f_{a,b}(q)$ of $[a,b]$ associated to the Eulerian kernel equals the Kazhdan--Lusztig polynomial $P_{a,b}(q)$. If $a<b$ and $D_{a,b}$ is any polytope with complete face lattice $[a,b]$, then
\begin{equation*}
 P_{a,b}(q)=g(D_{a,b}^*;q).
\end{equation*}
Consequently, these Kazhdan--Lusztig polynomials depend only on the isomorphism types of their ranked intervals.
\end{proposition}

\begin{proof}
Every subinterval of a lattice is a lattice. By Theorem~\ref{thm:lattice-characterization}, its $R$-polynomial is $(q-1)^{\ell(a,b)}$. The two kernels therefore agree on all comparable pairs in $L$. Uniqueness of the right KLS function gives $f_{a,b}(q)=P_{a,b}(q)$. If $a<b$ and $D_{a,b}$ is any polytope with complete face lattice $[a,b]$, then we have $f_{a,b}(q)=g(D_{a,b}^*;q)$ as discussed above. This proves the remaining assertions.
\end{proof}

This proposition separates two facts. The equality of these Kazhdan--Lusztig polynomials with the Eulerian KLS polynomials follows from the lattice condition and the $R$-polynomial characterization. Polytopality in our Theorem~\ref{thm:main} supplies the additional convex geometry needed to apply theorems about toric $g$-vectors. An arbitrary Eulerian poset has a KLS function, but this function need not satisfy the numerical restrictions on toric $g$-polynomials of convex polytopes discussed below.

Now we recall the precise polytope results that we will use. If $P$ is a convex $d$-polytope, then its toric $g$-numbers (i.e., coefficients of the toric $g$-polynomial) are nonnegative, and
\begin{align*}
 g(P;q)&\geq g(F;q)g(P/F;q)&&\text{ for any face $F$ of }P,\\
 g_k(P)=0&\ \Longrightarrow\ g_{k+1}(P)=0 && \text{ for $k\geq1$ },\\
 g_i(P)&\leq\binom{f_0(P)-d+i-2}{i}&&\text{ for $1\leq i\leq\lfloor d/2\rfloor$ }.
\end{align*}
Here $P/F$ is a face figure, with complete face lattice $[F,P]$, $g_i(P)$ is the coefficient of $q^i$ in $g(P;q)$ and $f_0(P)$ is the number of vertices of $P$. The first monotonicity inequality was conjectured by Kalai and was proved for rational polytopes by Braden--MacPherson \cite{BradenMacPherson1999}, the second implication is known as Kalai's vanishing, and the last inequality is Bayer's generalized upper bound theorem \cite{Bayer1998}. The combinatorial intersection cohomology of fans and Karu's hard Lefschetz theorem \cite{Karu2004} remove the rationality hypothesis. The statements, including Kalai's vanishing result, are recorded in \cite[Theorems~1.2--1.4]{Braden2006}.

We are now ready to establish the consequences of our main theorem, which provide additional information for the Kazhdan--Lusztig polynomials for lattice Bruhat intervals.

\begin{proof}[Proof of Theorem~\ref{thm:KL-applications}]
Let $P=L_{x,y}^{*}$, so that $\dim P=n-1$ and
\[
 P_{x,y}(q)=g(P;q).
\]
Then Kalai’s vanishing theorem proves (1).

For $x<z<y$, the face of $P$ corresponding to $z$ has complete face lattice anti-isomorphic to $[z,y]$, while its face figure in $P$ has complete face lattice anti-isomorphic to $[x,z]$. Their toric $g$-polynomials are $P_{z,y}$ and $P_{x,z}$ by Proposition~\ref{prop:kernel-coincidence}. Thus Kalai's monotonicity gives (2).

Note that the vertices of $P$ correspond to the coatoms of $[x,y]$. Substitution of $f_0(P)=m$ and $d=n-1$ into the generalized upper bound theorem proves (3).

Note that by Proposition~\ref{prop:kernel-coincidence}, the linear coefficient of the toric $g$-polynomial (see \cite[(5), p.~1153]{Braden2006}) gives the linear coefficient
\begin{equation*}
 [q]P_{x,y}(q)=f_0(P)-n=m-n.
\end{equation*}
Consequently $P_{x,y}(q)=1$ forces $P$ to have exactly $n$ vertices, and hence to be a simplex. Its face lattice, and therefore $[x,y]$, is Boolean. The converse follows from the fact that the toric $g$-polynomial of a simplex is $1$ (see \cite[p.~1153]{Braden2006}). This proves (4).
\end{proof}

\subsection*{Acknowledgment} The authors would like to thank Soyeon Kim for the wonderful talk in the CART seminar on unexpected Richardson varieties, which started this project. Both authors are partially supported by NSFC Grant no. 12471309.

\subsection*{Declaration of AI usage} AI tools, including GPT-5.6 Sol and GPT-6 Astra, were extensively used in the technical details of Theorem 1.1 through numerous rounds of interaction with the authors, while the remaining mathematical content was developed by the authors. The paper is written by the authors with limited AI-assisted language editing.

\bibliographystyle{plain}
\bibliography{template}
\end{document}